\documentclass[11pt]{amsart} 

\usepackage{amsmath, amssymb, amsfonts, tikz, hyperref, comment, url}
\usetikzlibrary{math}
\calclayout

\newtheorem{theorem}{Theorem}[section]
\newtheorem{lemma}[theorem]{Lemma}
\newtheorem{proposition}[theorem]{Proposition}
\newtheorem{corollary}[theorem]{Corollary}
\theoremstyle{definition}
\newtheorem{example}[theorem]{Example}
\newtheorem{remark}[theorem]{Remark}
\newtheorem{definition}[theorem]{Definition}
\newtheorem*{definition*}{Definition}
\newtheorem*{theorem*}{Theorem}
\newtheorem*{proposition*}{Proposition}
\newtheorem*{corollary*}{Corollary}

\renewcommand{\S}{\mathfrak{S}}
\newcommand{\Z}{\mathbb{Z}}

\author{Yibo Gao}
\address{Department of Mathematics, Massachusetts Institute of Technology, \mbox{Cambridge, MA 02139}}
\email{\href{mailto:gaoyibo@mit.edu}{{\tt gaoyibo@mit.edu}}}
\author{Junyao Peng}
\address{Department of Mathematics, Massachusetts Institute of Technology, \mbox{Cambridge, MA 02139}}
\email{\href{mailto:junyaop@mit.edu}{{\tt junyaop@mit.edu}}}

\begin{document}
\title{Counting Shellings of Complete Bipartite Graphs and Trees}
\date{\today}

\begin{abstract}
A shelling of a graph, viewed as an abstract simplicial complex that is pure of dimension 1, is an ordering of its edges such that every edge is adjacent to some other edges appeared previously. In this paper, we focus on complete bipartite graphs and trees. For complete bipartite graphs, we obtain an exact formula for their shelling numbers. And for trees, we relate their shelling numbers to linear extensions of tree posets and bound shelling numbers using vertex degrees and diameter.
\end{abstract}
\maketitle

\section{Introduction}
In combinatorial topology, shelling of a simplicial complex is a very useful and important notion that has been well-studied. 
\begin{definition}\label{def:shellingcomplex}
An (abstract) simplicial complex $\Delta$ is called \textit{pure} if all of its maximal simplicies have the same dimension. Given a finite (or countably infinite) simplicial complex $\Delta$ that is pure of dimension $d$, a \textit{shelling} is a total ordering of its maximal simplicies $C_1,C_2,\ldots$ such that for every $k>1$, $C_k\cap\left(\bigcup_{i=1}^{k-1}C_i\right)$ is pure of dimension $d-1$. A simplicial complex that admits a shelling is called \textit{shellable}.
\end{definition}
Shellable complexes enjoy many strong algebraic and topological properties. For example, a shellable complex is homotopy equivalent to a wedge sum of spheres, thus its homology groups (over $\Z$) are torsion-free. The study of shellability in its combinatorial aspects has turned out to be very fruitful as well. The arguably earliest notable result that polytopes are shellable is due to Brugesser and Mani (Section 8 of \cite{ziegler2012lectures}). Later on, Bj\" orner and Wachs developed theories on lexicographic
shellability (Section 12 of \cite{kozlov2007combinatorial}). In particular, they introduced powerful notions of $EL$- and $CL$-shellability to study graded posets whose order complexes are shellable. In a recent work, Goaoc et al. \cite{goaoc2017shellability} proved that testing shellability is NP-complete.

A great deal of work has been done in understanding properties of shellable complexes and identifying classes of shellable complexes, however, little work has been done in counting the number of shellings of a fixed simplicial complex. One result on the number of shellings of the complete graphs is given by Stanley \cite{MO297411}. It is generally believed that if a simplicial complex is shellable, then it usually admits a lot of shellings, but no precise arguments have been given.

In this paper, we investigate the problem of counting shellings, aiming to start a new line of research. We restrict our attention to finite simplicial complexes that are pure of dimension 1, namely, undirected graphs, where interesting combinatorial arguments are already taking place. Let's first reformulate Definition~\ref{def:shellingcomplex} in the language of graph theory.
\begin{definition}[Graph Shelling]\label{def:shellinggraph}
Given an undirected graph $G=(V,E)$, where $V$ is the vertex set of $G$ and $E$ is the edge set of $G$, a \textit{shelling} of $G$ is a total ordering of the edge set $\sigma\in\S_E$, where $\S$ stands for symmetric group, such that $\sigma(1),\ldots,\sigma(k)$ form a connected subgraph of $G$ for all $k=1,\ldots,|E|$.
\end{definition}
Throughout this paper, we will use $F(G)$ to denote the number of shellings of a graph $G$.

Clearly, a graph admits a shelling if and only if it is connected, which is equivalent to $F(G)>0$. A few results are already known.
\begin{theorem}[\cite{MO297411}]
Let $K_n$ be the complete graph on $n$ vertices. Then
$$F(K_n)=\frac{2^{n-2}}{C_{n-1}}\binom{n}{2}!$$
where $C_{n-1}=\binom{2n-2}{n-1}/n$ is the $(n-1)^{th}$ Catalan number.
\end{theorem}

As an overview for the paper, in Section~\ref{sec:cbg}, we will give an explicit formula for the number of shellings of complete bipartite graphs, resolving a MathOverflow question \cite{MO297385} posed by Sebastien Palcoux; in Section~\ref{sec:trees}, we will provide exact formula for the number of shellings of trees, which is closely related to the number of linear extensions of posets, and utilize the formula to obtain some upper and lower bounds for them.
\section{Complete Bipartite Graphs}\label{sec:cbg}
Denote $K_{m,n}$ as the complete bipartite graph with part sizes $m$ and $n$. The following is our main theorem.
\begin{theorem}\label{thm:bipartite}
For all positive integers $m$ and $n$, the number of shellings of $K_{m,n}$ is given by
$$F(K_{m,n})=\frac{m!n!(mn)!}{(m+n-1)!}.$$
\end{theorem}
\begin{remark}
Theorem~\ref{thm:bipartite} can also be rephrased as $F(K_{m,n})=(mn)!\cdot(m+n)/{m+n\choose n}$. In other words, there is a $(m+n)/{m+n\choose n}$ probability that a uniformly random edge ordering is a shelling order.
\end{remark}
The formula in Theorem~\ref{thm:bipartite} is conjectured in the MathOverflow post \cite{MO297385}. Partial progress has been made: Lemma~\ref{lem:bipartiteStanley}, given by Richard Stanley \cite{MO297385}, serves as an important tool for our computation.

\begin{lemma}\label{lem:bipartiteStanley}
$F(K_{m,n})$ is equal to the following expression:
$$m!n!(mn-1)!\sum_{\alpha}\frac{b_{1}b_{2}\cdots b_{m+n-2}}{b_{m+n-2}(b_{m+n-2}+b_{m+n-3})\cdots (b_{m+n-2}+b_{m+n-3}+\ldots+b_1)},$$
where the sum is over all sequences $\alpha = (a_1,a_2,\ldots,a_{m+n-2})$ of $(m-1)$ 0's and $(n-1)$ 1's, and
$$b_{i} = 1 + |\{1\leq j\leq i: a_j \neq a_i\}|.$$
\end{lemma}
\begin{proof}
Let $\sigma$ be a shelling of $K_{m,n}$. In each part of $K_{m,n}$, consider the order of the appearance of the vertices. Here, we say that vertex $v$ appears in $\sigma$ at time $t$ if $t$ is the first index such that $v\in \sigma(t)$. There are $m!$ ways to choose such order in the part of size $m$ and $n!$ ways in the part of size $n$. Fix the order of vertex appearance in each part to be $(u_0,u_1,\ldots,u_{m-1})$ and $(v_0,v_1,\ldots,v_{n-1})$, respectively.

Consider a fixed order of appearance of all $(m+n)$ vertices $w = w_{-1}w_0\ldots w_{m+n-2}.$ Note that $\sigma(1)$ must be the edge $e_0=(u_0,v_0)$, so $\{w_{-1},w_0\} = \{u_0,v_0\}.$ For $1\leq i\leq m+n-2$, define 
$$a_{i} = 
\begin{cases}
0, & \text{ if } w_{i} = u_j \text{ for some }j, \\
1, & \text{ if } w_{i} = v_k \text{ for some }k,
\end{cases}$$
and 
$$b_i = 1 + |\{1\leq j\leq i: a_j \neq a_i\}|.$$
Now, for each $w_i$ ($i\geq 1$), consider the first edge $e_i$ incident to $w_i$ in $\sigma$. This edge must be of the form $(w_i,w_j)$ where $j<i$ and $w_i, w_j$ are in different parts of $K_{m,n}$. There are $b_{i}$ choices for this edge, where the 1 in the definition of $b_i$ refers to the edge connecting to either $u_0$ or $v_0$. Thus, there are $b_1b_2\cdots b_{m+n-2}$ ways to choose $e_1,e_2,\ldots,e_{m+n-2}.$

Fix these edges $e_0,e_1,\ldots,e_{m+n-2}$. Note that the rest of the $b_{m+n-2}$ edges incident to $w_{m+n-2}$ must appear after $e_{m+n-2}$ in $\sigma$, so there are $(b_{m+n-2}-1)!$ ways to arrange these edges. After making this arrangement, the edges which are incident to $w_{m+n-3}$ and not yet arranged must appear after $e_{m+n-3}$, so there are
$$(b_{m+n-2}+1)(b_{m+n-2}+2)\cdots (b_{m+n-2}+b_{m+n-3}+1) = \frac{(b_{m+n-2}+b_{m+n-3}+1)!}{b_{m+n-2}!}$$
ways to arrange them (since there are already $b_{m+n-2}$ edges arranged after $e_{m+n-3}$). Similarly, for each $i$, after making the arrangement of all edges incident to vertices appearing after $w_i$, there are 
$$\frac{(b_{m+n-2}+b_{m+n-3}+\ldots + b_i + 1)!}{(b_{m+n-2}+b_{m+n-3}+\ldots + b_{i+1})!}$$
ways to arrange all the edges which are incident to $w_i$ and not yet arranged. Therefore, after fixing $e_0,e_1,\ldots,e_{m+n-2}$, the number of shellings is
$$\prod_{i=1}^{m+n-2}\frac{(b_{m+n-2}+\ldots + b_i + 1)!}{(b_{m+n-2}+\ldots + b_{i+1})!} = \frac{(mn-1)!}{b_{m+n-2}(b_{m+n-2}+b_{m+n-3})\cdots (b_{m+n-2}+\ldots+b_1)}.$$

Combining all discussions above, we obtain Lemma~\ref{lem:bipartiteStanley}.
\end{proof}

We will now compute the expression in Lemma~\ref{lem:bipartiteStanley} explicitly to finish the proof of Theorem~\ref{thm:bipartite}.
\begin{proof}[Proof of Theorem~\ref{thm:bipartite}]
For $m,n\in\Z_{\geq0}$, let's define the following sum
$$S(p,q,m,n):=\sum_{\alpha}\frac{b_1b_2\cdots b_{m+n}}{b_{m+n}(b_{m+n}+b_{m+n-1})\cdots(b_{m+n}+\cdots+b_1)},$$
over all sequences $\alpha = (a_1,a_2,\ldots,a_{m+n})$ of $m$ 0's and $n$ 1's, and
$$b_{i} = \begin{cases}
q + |\{1\leq j\leq i: a_j=1\}|,&\quad a_i=0\\
p + |\{1\leq j\leq i: a_j=0\}|,&\quad a_i=1
\end{cases}.$$
Visually (see Figure~\ref{fig:visualS}), the sum in $S(p,q,m,n)$ goes through all lattice paths from $(0,0)$ to $(m,n)$, and $b_i$ is the length of the (horizontal or vertical) strip formed by the $i^{th}$ step of the path with the boundary $x=-q$ and $y=-p$. 
\begin{figure}[h!]
\centering
\begin{tikzpicture}[scale=0.7]
\tikzmath{\k=0.5;}
\draw(0,0)--(1,0)--(1,1)--(3,1)--(3,2)--(5,2)--(5,4);
\node at (0,0) {$\bullet$};
\node at (1,0) {$\bullet$};
\node at (1,1) {$\bullet$};
\node at (2,1) {$\bullet$};
\node at (3,1) {$\bullet$};
\node at (3,2) {$\bullet$};
\node at (4,2) {$\bullet$};
\node at (5,2) {$\bullet$};
\node at (5,3) {$\bullet$};
\node at (5,4) {$\bullet$};
\draw[dashed](-3,-2)--(3,-2);
\draw[dashed](-3,-2)--(-3,2);
\draw[dashed](0,0)--(0,-2);
\draw[dashed](0,0)--(-3,0);
\draw[dashed](1,0)--(1,-2);
\draw[dashed](1,1)--(-3,1);
\draw[dashed](2,1)--(2,-2);
\draw[dashed](3,1)--(3,-2);
\draw[dashed](3,2)--(-3,2);
\node at (0.5,-1) {$b_1$};
\node at (-1,0.5) {$b_2$};
\node at (1.5,-0.5) {$b_3$};
\node at (2.5,-0.5) {$b_4$};
\node at (0,1.5) {$b_5$};
\node at (4,0) {$\cdots$};
\node at (1.5,3) {$\vdots$};

\draw (5+\k,4)--(5+\k,2.2);
\draw (5+\k,1.8)--(5+\k,0);
\draw (5+\k-0.2,0)--(5+\k+0.2,0);
\draw (5+\k-0.2,4)--(5+\k+0.2,4);
\node at (5+\k,2) {$n$};
\draw (0,4+\k)--(2.2,4+\k);
\draw (2.8,4+\k)--(5,4+\k);
\draw (0,4+\k-0.2)--(0,4+\k+0.2);
\draw (5,4+\k-0.2)--(5,4+\k+0.2);
\node at (2.5,4+\k) {$m$};
\draw (-3,-2-\k)--(-1.7,-2-\k);
\draw (-1.3,-2-\k)--(0,-2-\k);
\draw (-3,-2-\k+0.2)--(-3,-2-\k-0.2);
\draw (0,-2-\k+0.2)--(0,-2-\k-0.2);
\node at (-1.5,-2-\k) {$p$};
\draw (-3-\k,0)--(-3-\k,-0.7);
\draw (-3-\k,-1.3)--(-3-\k,-2);
\draw (-3-\k-0.2,0)--(-3-\k+0.2,0);
\draw (-3-\k-0.2,-2)--(-3-\k+0.2,-2);
\node at (-3-\k,-1) {$q$};

\node at (8,3) {$\bullet$};
\draw[->](8,3)--(8,4);
\draw[->](8,3)--(9,3);
\node at (8.3,4) {$1$};
\node at (9,3.3) {$0$};
\end{tikzpicture}
\caption{A visualization for the sum $S(p,q,m,n)$}
\label{fig:visualS}
\end{figure}

By Lemma~\ref{lem:bipartiteStanley}, it suffices to compute $S(1,1,m-1,n-1)$. As a piece of notation, for $n\in\Z_{\geq0}$, write $t^{(n)}$ to denote $t\cdot (t+1)\cdots (t+n-1)$ where $t^{(0)}=1$. We are now going to use induction on $m+n$ to show the following claim:
$$S(p,q,m,n)=\frac{1}{m!n!}\frac{(p+q)^{(m)}(p+q)^{(n)}}{(p+q)^{(m+n)}}.$$
The base case is $m=0$ or $n=0$. If $m=0$, then the sum in $S(p,q,m,n)$ has only one term with $b_1=b_2=\cdots=b_n=q$. We then obtain $S(p,q,0,n)=1/n!$ as desired. The case $n=0$ is similar. Now fix $m,n>0$ and assume that the claim is true for $S(p,q,m',n')$, where $m'+n'<m+n$ and $p,q$ are arbitrary. 

Consider $S(p,q,m,n)$. For a sequence $\alpha=(\alpha_1,\ldots,\alpha_{m+n})$ in the summand, either $\alpha_1=0$ or $\alpha_1=1$. Notice that $b_{m+n}+\cdots+b_1=(m+p)(n+q)-pq=mn+pn+qm$. If $\alpha_1=0$, then $b_1=q$ and $$\frac{b_2b_3\cdots b_{m+n}}{b_{m+n}(b_{m+n}+b_{m+n-1})\cdots(b_{m+n}+\cdots+b_2)}$$
is a summand in $S(p+1,q,m-1,n)$. We can similarly analyze the case $\alpha_1=1$ and obtain the following recursive formula
\begin{align*}
S(p,q,m,n)=&\frac{q}{mn+pn+qm}S(p+1,q,m-1,n)\\
&+\frac{p}{mn+pn+qm}S(p,q+1,m,n-1).
\end{align*}
By the induction hypothesis, the above calculation goes on as follows
\begin{align*}
=&\frac{qS(p+1,q,m-1,n)+pS(p,q+1,m,n-1)}{mn+pn+qm}\\
=&\frac{qm(p+q+1)^{(m-1)}(p+q+1)^{(n)}+pn(p+q+1)^{(m)}(p+q+1)^{(n-1)}}{(mn+pn+qm)m!n!(p+q+1)^{(m+n-1)}}\\
=&\frac{qm(p+q)^{(m)}(p+q+1)^{(n)}+pn(p+q+1)^{(m)}(p+q)^{(n)}}{(mn+pn+qm)m!n!(p+q)^{(m+n)}}\\
=&\frac{(p+q)^{(m)}(p+q)^{(n)}}{m!n!(p+q)^{(m+n)}}\frac{1}{mn+pn+qm}\left(\frac{(qm)(p+q+n)}{p+q}+\frac{(pn)(p+q+m)}{p+q}\right)\\
=&\frac{(p+q)^{(m)}(p+q)^{(n)}}{m!n!(p+q)^{(m+n)}}.
\end{align*}
Therefore, the induction step goes through and the claim is established.

In particular, by Lemma~\ref{lem:bipartiteStanley},
\begin{align*}
F(K_{m,n})=&m!n!(mn-1)!S(1,1,m-1,n-1)\\
=&m!n!(mn-1)!\frac{m!n!}{(m-1)!(n-1)!(m+n-1)!}\\
=&\frac{m!n!(mn)!}{(m+n-1)!}.
\end{align*}
\end{proof}
\begin{remark}
An earlier version of this paper goes from Lemma~\ref{lem:bipartiteStanley} to Theorem~\ref{thm:bipartite} via heavy calculation of algebraic expressions. The current proof presented here is much shorter but a bijective proof would be very desirable.
\end{remark}

\section{Trees}\label{sec:trees}
Trees are one of the most fundamental types of graphs. However, unlike the complete bipartite graph case, there is no simple formula for tree shelling numbers. The goal of this section is to give a relatively easy method to compute the number of shellings of a tree and to provide upper and lower bounds of this number.

In the process, we will see intimate connections between the number of shellings of a tree and the number of linear extensions of tree posets, which is a much better explored subject in the literature.

Throughout this section, let $T$ be a tree with $n$ vertices and $n-1$ edges.

\subsection{Tree Shelling Number Computation}
We first focus on computing the number of shellings of rooted trees, whose definition is given below.

\begin{definition}\label{def:shellingOfRootedTree}
Let $v$ be a vertex of $T$. The rooted tree induced by $T$ and rooted at $v$ is denoted as $T_v$. A \textit{shelling of a rooted tree} $T_v$ is a shelling $\sigma$ of $T$ such that $\sigma(1)$ is an edge incident to $v$.
\end{definition}

The following definitions are used to efficiently describe structures in a (rooted) tree.

\begin{definition}\label{def:parent}
Let $T_v$ be a tree rooted at vertex $v$.  We say a vertex $u$ is a \textit{parent} of vertex $w$ (and $w$ is a \textit{child} of $u$) if $(w,u)$ is an edge and $u$ lies closer to the root than $w$. A \textit{descending path} from $u$ to $w$ in the rooted tree $T_v$ is a structure
$$u-v_1-v_2-\cdots -v_r-w$$
where each vertex is a parent of the subsequent vertex. We say $u$ is an \textit{ancestor} of $w$ (and $w$ is a \textit{descendant} of $u$) if there exists a descending path from $u$ to $w$.
\end{definition}

\begin{definition}\label{def:rootedSubtree}
Let $u,v\in T.$ The (rooted)\textit{ subtree of $T_v$ rooted at $u$}, denoted as $T_v(u)$, is a subgraph of $T$ rooted at $u$ and induced by the set of vertices
$$\{w\in T: w \text{ is a descendant of } u \text{ in }T_v\}.$$
See Figure~\ref{fig:definition} for an example.
\end{definition}
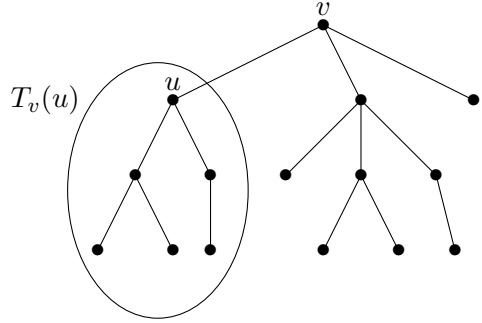
\begin{figure}[h]
\begin{tikzpicture}[scale=1]
\draw (2,-1)--(0,0)--(-2,-1);
\draw (0,0)--(0.5,-1);
\draw (-1.5,-2)--(-2,-1)--(-2.5,-2);
\draw (0.5,-2)--(0.5,-1)--(1.5,-2);
\draw (-0.5,-2)--(0.5,-1);
\draw (-3,-3)--(-2.5,-2)--(-2,-3);
\draw (-1.5,-3)--(-1.5,-2);
\draw (0,-3)--(0.5,-2)--(1,-3);
\draw (1.75,-3)--(1.5,-2);

\node at (0,0){$\bullet$};
\node at (-2,-1){$\bullet$};
\node at (2,-1){$\bullet$};
\node at (0.5,-1){$\bullet$};
\node at (-2.5,-2){$\bullet$};
\node at (-1.5,-2){$\bullet$};
\node at (0.5,-2){$\bullet$};
\node at (-0.5,-2){$\bullet$};
\node at (1.5,-2){$\bullet$};
\node at (-3,-3){$\bullet$};
\node at (-2,-3){$\bullet$};
\node at (-1.5,-3){$\bullet$};
\node at (1,-3){$\bullet$};
\node at (0,-3){$\bullet$};
\node at (1.75,-3){$\bullet$};

\draw (0,0) node[above]{$v$};
\draw (-2,-1) node[above]{$u$};
\draw (-2.2,-2.2) ellipse (1.2 and 1.7);
\draw (-3.1,-1) node[left]{$T_v(u)$};
\end{tikzpicture}
\caption{Definition of $T_v(u)$.}
\label{fig:definition}
\end{figure}

For a tree $T$, the edge set of $T$ is denoted as $E(T)$. The vertex set of $T$ is denoted as $V(T)$, or $T$ for simplicity. Accordingly, $|T|$ is the number of vertices in $T$. The same notations are used for rooted trees.

The following proposition provides a way to calculate the number of shellings of a rooted tree $T_v$ based on the size of its rooted subtrees.

\begin{proposition}\label{prop:rootedTreeCounting}
Let $T$ be a tree on $n$ vertices and $v\in V(T)$. Then
$$F(T_v) = \frac{n!}{\prod_{u\in T} |T_v(u)|}.$$
\end{proposition}

Proposition~\ref{prop:rootedTreeCounting} is, in fact, Knuth's hook length formula \cite{knuth1973art} for the number of linear extensions of a partially ordered set whose Hasse diagram is the rooted tree $T_v$. Recall that a linear extension of a poset $P$ on $n$ elements is a bijection $\beta:P\rightarrow\{1,2,\ldots,n\}$ such that $\beta(x)<\beta(y)$ if $x<y$ in $P$. Specifically, shellings of $T_v$ are in simple bijection with linear extensions of the poset $P_v$ where $u<w$ if $u$ is a descendent of $w$ in $T_v$: given a shelling $\alpha$ of $T_v$, we can construct a linear extension $\beta:P\rightarrow\{1,\ldots,n\}$ such that $\beta(v)=n$ and $\beta^{-1}(n-1),\ldots,\beta^{-1}(1)$ is the order of appearances of vertices $P_v\setminus\{v\}$ provided by $\alpha$. 

Because of this correspondence, we omit the proof of Proposition~\ref{prop:rootedTreeCounting}, which is in the literature and is straightforward via induction. Knuth's hook length formula for trees has multiple $q$-analogues given by Bj\"{o}rner and Wachs \cite{bjorner1989qhook} and a further multivariate generalization by Hivert and Reiner \cite{hivert2013multivariate}. Viewing such linear extensions of tree posets as intervals in the weak Bruhat orders, we see that those intervals are called ``forest quotients in the symmetric group" in the language of generalized quotients developed by Bj\"{o}rner and Wachs \cite{bjorner1988generalized} and that they are nice intervals to provide splittings of the Coxeter groups \cite{gaetz2019sep2,gaetz2020sep1}.

For our purposes to compute $F(T)$, however, the root of the tree is constantly moving. Thus, we need new tools for our estimation.

\begin{corollary}\label{cor:rootedTreeRatio}
Suppose that $(u,v)$ is an edge of $T$, then
$$\frac{F(T_v)}{F(T_u)} = \frac{|T_u(v)|}{|T_v(u)|} = \frac{|T_u(v)|}{n-|T_u(v)|}.$$
\end{corollary}
\begin{proof}
For any vertex $w\neq u,v$, $T_u(w)$ and $T_v(w)$ are the same subtree of $T_v$. Therefore, by Proposition~\ref{prop:rootedTreeCounting},
\begin{align*}
\frac{F(T_v)}{F(T_u)} &= \frac{\prod_{w\in T} |T_u(w)|}{\prod_{w\in T} |T_v(w)|} = \frac{|T_u(v)|\cdot |T_u(u)|}{|T_v(u)|\cdot |T_v(v)|}  \\
&= \frac{|T_u(v)|}{|T_v(u)|}  = \frac{|T_u(v)|}{n-|T_u(v)|}.
\end{align*}
\end{proof}

Corollary ~\ref{cor:rootedTreeRatio} establishes a simple relationship between the number of shellings of $T$ rooted at adjacent vertices. In this way, by only calculating $F(T_v)$ for a single vertex $v$, one can quickly derive $F(T_u)$ for all $u\in T.$ For example, suppose $T$ is a path of length $n-1$, as shown in figure~\ref{fig:path}. Then $F(T_{v_1}) = 1$, and $$F(T_{v_{i+1}}) = \frac{|T_{v_i}(v_{i+1})|}{n-|T_{v_i}(v_{i+1})|}F(T_{v_i}) = \frac{n-i}{i}F(T_{v_i})$$ by Corollary~\ref{cor:rootedTreeRatio}. This gives $F(T_{v_i}) = \binom{n-1}{i-1}$ for all $i=1,2,\ldots,n.$

\begin{figure}[h]
\begin{tikzpicture}[scale=1]
\draw(0,0)--(4,0);
\draw(5,0)--(7,0);

\node at (0,0){$\bullet$};
\node at (1,0){$\bullet$};
\node at (2,0){$\bullet$};
\node at (3,0){$\bullet$};

\node at (7,0){$\bullet$};
\node at (6,0){$\bullet$};

\node[draw=none] (ellipsis2) at (4.5,0) {$\cdots$};

\draw (0,0) node[below]{$v_1$};
\draw (1,0) node[below]{$v_2$};
\draw (2,0) node[below]{$v_3$};
\draw (3,0) node[below]{$v_4$};
\draw (6,0) node[below]{$v_{n-1}$};
\draw (7,0) node[below]{$v_n$};
\end{tikzpicture}
\caption{A path of length $n-1$. The shelling number is $2^{n-2}$.}
\label{fig:path}
\end{figure}
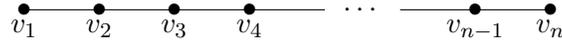

Finally, the following proposition relates the number of shellings of $T$ with that of its rooted trees.

\begin{proposition}\label{prop:treeCounting}
$$F(T) = \frac{1}{2}\sum_{v\in T} F(T_v).$$
\end{proposition}
\begin{proof}
Note that any shelling of $T$ beginning with edge $(u,v)$ is counted as a shelling of both $T_u$ and $T_v$. Thus, Proposition~\ref{prop:treeCounting} follows.
\end{proof}
Recall that $F(T_v)=e(P_v)$ where $P_v$ is the poset where $u<w$ if $u$ is a descendent of $w$ and $e$ denotes the number of linear extensions of a poset. In other words, Proposition~\ref{prop:treeCounting} is saying that $F(T)=\frac{n}{2}\mathbb{E}[e(P_v)]$ where the root is chosen uniformly at random. Equivalently, in the rest of the paper, we will be providing bounds for the expectation $\mathbb{E}[e(P_v)]$, which is a novel statistic on trees.

\begin{example}\label{ex:path}
By Proposition~\ref{prop:treeCounting} and the discussion under Corollary~\ref{cor:rootedTreeRatio}, the number of shellings of a path of length $n-1$ is
$$\frac{1}{2}\sum_{i=1}^n \binom{n-1}{i-1} = 2^{n-2}.$$
\end{example}

\subsection{Bounds on Tree Shelling Number}
The goal of this section is to give several bounds of tree shelling numbers based on various parameters of a graph, such as vertex degree and diameter (the diameter of a connected graph is the greatest distance between any pairs of vertices). A trivial upper bound is $(n-1)!$, since every shelling is also a permutation of edges. The upper bound is achieved when $T$ is a star, in which every two edges are adjacent to each other.

Here are the main theorems of the section.

\begin{theorem}\label{thm:lowerBoundDegree}
Let $T$ be a tree. For each $v\in V(T)$, let $d(v)$ denote its degree. Then
\begin{equation*}
F(T) \geq \prod_{v\in T} d(v)!.
\end{equation*}
The equality holds if and only if $T$ is a path of length $n-1$ or a star.
\end{theorem}
\begin{remark}
A weaker lower bound $F(T)\geq\prod_{v\in T}\big(d(v)-1\big)!$ is easy to see. However, an extra factor of $\prod_{v\in T}d(v)$ in Theorem~\ref{thm:lowerBoundDegree} requires much more effort.
\end{remark}

\begin{theorem}\label{thm:upperBoundDiameter}
Suppose the diameter of $T$ is $\ell$. When $\ell$ is even,
$$F(T)\leq \frac{2(n-1-\frac{\ell}{2})!}{(\frac{\ell}{2})!}\bigg[\binom{n-2}{\frac{\ell}{2}}+\sum_{i=0}^{\frac{\ell}{2}-1}\binom{n-1}{i}\bigg].$$
When $\ell$ is odd,
$$F(T) \leq \frac{(n-\frac{\ell+3}{2})!}{(\frac{\ell+1}{2})!}\bigg[(n-1-\ell)\binom{n-2}{\frac{\ell-1}{2}}+n\sum_{i=0}^{\frac{\ell-1}{2}}\binom{n-1}{i}\bigg].$$
The equality holds if and only if $T$ has the following form: there exists a path $$v_0 - v_1 - \cdots -v_\ell$$ such that every edge not in this path is adjacent to $v_{\lfloor \frac{\ell}{2}\rfloor}.$
\end{theorem}

Before proving Theorem~\ref{thm:lowerBoundDegree}, it is worth noticing the following inequality, which relates the number of shellings of $T$ and $T_v$.

\begin{lemma}\label{lem:weightInequality}
Let $v$ be a vertex in $T$ and $\ell$ be the length of the longest descending path in $T_v$. Then
$$F(T)\leq \bigg[\sum_{k=0}^{\ell-1}\binom{n-2}{k}\bigg] F(T_v).$$
In particular, $F(T)\leq 2^{n-2}F(T_v).$
\end{lemma}
\begin{proof}
Let $L = v - v_1 - v_2 -\cdots - v_\ell$ be the longest descending path in $T_v$. Perform the following operation on $T$ until it can be performed no longer:
\begin{enumerate}
\item[] Suppose $i\leq \ell-2$ is the first index such that $v_i$ has a child $v'\neq v_{i+1}$ in $T_v$. Remove $T_v(v')$ and attach it on $v_{i+1}$ (i.e., children of $v'$ become children of $v_{i+1}$). Furthermore, remove edge $(v',v_i)$ and add a new edge $(v',v_{i+1}).$ This operation is illustrated in Figure~\ref{fig:pathAdjustmentForWeight}. performed.
\end{enumerate} 

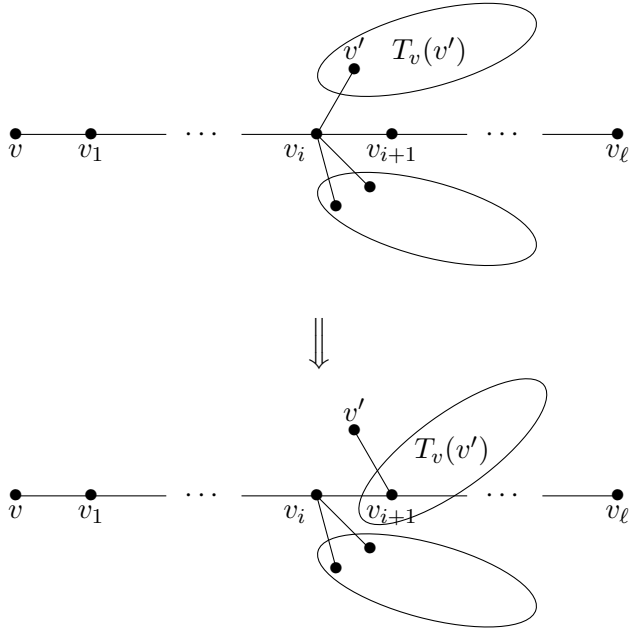
\begin{figure}[h]
\begin{tikzpicture}[scale=1]
\draw(0,0)--(2,0);
\draw(3,0)--(6,0);
\draw(7,0)--(8,0);
\draw(4,0)--(4.5,0.866);
\draw(4,0)--(4.707,-0.707);
\draw(4,0)--(4.259,-0.966);

\node at (0,0){$\bullet$};
\node at (1,0){$\bullet$};
\node at (4,0){$\bullet$};
\node at (5,0){$\bullet$};
\node at (4.5,0.866){$\bullet$};
\node at (4.707,-0.707){$\bullet$};
\node at (4.259,-0.966){$\bullet$};
\node at (8,0){$\bullet$};

\draw[rotate around={-15:(4.5,-0.866)}] (5.5,-0.866) ellipse (1.5 and 0.5);
\draw[rotate around={15:(4.5,0.866)}] (5.5,0.866) ellipse (1.5 and 0.5);

\node[draw=none] (ellipsis2) at (2.5,0) {$\cdots$};
\node[draw=none] (ellipsis2) at (6.5,0) {$\cdots$};

\draw (0,0) node[below]{$v$};
\draw (1,0) node[below]{$v_1$};
\draw (4,0) node[below left]{$v_i$};
\draw (5,0) node[below]{$v_{i+1}$};
\draw (4.5,0.866) node[above]{$v'$};
\draw (5.5,1.1) node[]{$T_{v}(v')$};
\draw (8,0) node[below]{$v_\ell$};
\end{tikzpicture}

$$\Big\Downarrow$$

\begin{tikzpicture}[scale=1]
\draw(0,0)--(2,0);
\draw(3,0)--(6,0);
\draw(7,0)--(8,0);
\draw(5,0)--(4.5,0.866);
\draw(4,0)--(4.707,-0.707);
\draw(4,0)--(4.259,-0.966);

\node at (0,0){$\bullet$};
\node at (1,0){$\bullet$};
\node at (4,0){$\bullet$};
\node at (5,0){$\bullet$};
\node at (4.5,0.866){$\bullet$};
\node at (4.707,-0.707){$\bullet$};
\node at (4.259,-0.966){$\bullet$};
\node at (8,0){$\bullet$};

\draw[rotate around={-15:(4.5,-0.866)}] (5.5,-0.866) ellipse (1.5 and 0.5);
\draw[rotate around={36:(5,0)}] (6,0) ellipse (1.5 and 0.5);

\node[draw=none] (ellipsis2) at (2.5,0) {$\cdots$};
\node[draw=none] (ellipsis2) at (6.5,0) {$\cdots$};

\draw (0,0) node[below]{$v$};
\draw (1,0) node[below]{$v_1$};
\draw (4,0) node[below left]{$v_i$};
\draw (5,0) node[below]{$v_{i+1}$};
\draw (4.5,0.866) node[above]{$v'$};
\draw (5.8,0.6) node[]{$T_{v}(v')$};
\draw (8,0) node[below]{$v_\ell$};
\end{tikzpicture}
\caption{Operation on $T$: moving edges away from the root.}
\label{fig:pathAdjustmentForWeight}
\end{figure}

Such operations preserve the length of the longest descending path in $T_v$ and would eventually stop within a finite number of steps: the sum of distances from all vertices to $v_{\ell}$ strictly decreases after each step. Let $T^{(k)}$ be the tree after $k$th operation. For each $u\in T$, define the weight of $u$ in $T^{(k)}$
$$W_k(u) = \frac{F(T^{(k)}_u)}{F(T^{(k)}_v)}.$$
We claim that the sum of weights of all vertices is non-decreasing after each operation, i.e.
\begin{equation}\label{eq:weightInequality}
\sum_{u\in T}W_k(u) \leq \sum_{u\in T}W_{k+1}(u).
\end{equation}
It suffices to prove the claim for $k=0.$
By Corollary~\ref{cor:rootedTreeRatio}, suppose $(u,w)$ is an edge in $T^{(0)}$, then
\begin{equation}\label{eq:weightRatio}
\frac{W_0(u)}{W_0(w)} = \frac{|T^{(0)}_w(u)|}{|T^{(0)}_u(w)|} = \frac{|T^{(0)}_w(u)|}{n-|T^{(0)}_w(u)|}.
\end{equation}
Therefore, suppose $v-u_1-u_2-\cdots - u_r = u$ is a path in $T^{(0)}_v$, then
\begin{equation*}
W_0(u) = \prod_{j=1}^r \frac{|T^{(0)}_v(u_j)|}{n-|T^{(0)}_v(u_j)|}.
\end{equation*}
Note that for all $u\not\in \{v', v_{i+1}\}$, $|T^{(1)}_v(u)| = |T_v(u)|$. For each $w\not\in T_v(v')\cup T_v(v_{i+1})$, $v',v_{i+1}$ are not on the path from $v$ to $w$, so
\begin{equation*}
W_0(w) = W_1(w).
\end{equation*}
Write $|T_v(v'))|=a, |T_v(v_{i+1})| = b$, then $|T^{(1)}_v(v')| = 1$, $|T^{(1)}_v(v_{i+1})| = |T_v(v')|+|T_v(v_{i+1})| = a+b.$ By (\ref{eq:weightRatio}),
$$W_0(v') = \frac{a}{n-a}W_0(v_i).$$
$$W_0(v_{i+1}) = \frac{b}{n-b}W_0(v_i).$$
$$W_1(v_{i+1}) = \frac{a+b}{n-a-b}W_1(v_i).$$
Since $W_0(v_i) = W_1(v_i)$,
\begin{equation*}
W_0(v')+W_0(v_{i+1})\leq W_1(v_{i+1}).
\end{equation*}
For each $w\in T_v(v')\setminus\{v'\}$, by (\ref{eq:weightRatio}), 
\begin{equation*}
\frac{W_1(w)}{W_1(v_{i+1})} = \frac{W_0(w)}{W_0(v')} \implies W_0(w)\leq W_1(w).
\end{equation*}
Similarly, for each $w\in T_v(v_{i+1})\setminus\{v_{i+1}\}$,
\begin{equation*}
\frac{W_1(w)}{W_1(v_{i+1})} = \frac{W_0(w)}{W_0(v_{i+1})}\implies W_0(w)\leq W_1(w).
\end{equation*}
Therefore, we conclude that
$$\sum_{w\in T} W_0 (w)\leq \sum_{w\in T} W_1(w),$$
and (\ref{eq:weightInequality}) is proved.

Finally, suppose the operation stops after step $M$, then $T^{(M)}$ is the tree where all vertices not in $L$ are incident to $v_{\ell-1}$. Thus, by (\ref{eq:weightRatio}),
\begin{align*}
\sum_{u\in T} W_{M}(u) &= W_M(v)+W_M(v_1)+\cdots+W_M(v_{\ell-1}) + (n-\ell)W_M(v_\ell) \\
&=\sum_{i=0}^{\ell-1}\binom{n-1}{i} + (n-\ell)\frac{\binom{n-1}{\ell-1}}{n-1}\\
&=2\sum_{i=0}^{\ell-1}\binom{n-2}{i},
\end{align*}
where the last equality uses the identity
\[
\binom{n-1}{i} = \binom{n-2}{i} + \binom{n-2}{i-1}
\]
for all $i\geq 1$.
According to equation (\ref{eq:weightInequality}), Proposition~\ref{prop:treeCounting}, 
$$\frac{F(T)}{F(T_v)} = \frac{1}{2}\sum_{u\in T}W_0(u) \leq \frac{1}{2}\sum_{u\in T}W_M(u) = \sum_{i=0}^{\ell-1}\binom{n-2}{i},$$
so the proof is complete.
\end{proof}

Now we are ready to prove Theorem~\ref{thm:lowerBoundDegree} and~\ref{thm:upperBoundDiameter}.
\begin{proof}[Proof of Theorem~\ref{thm:lowerBoundDegree}]
Induct on $|V(T)|$. When $|V(T)| = 2$, $F(T) = 2 = \prod_{v\in T} d(v)!$. In this case, $T$ is both a path and a star, which justifies the case of equality. 

Let $n>2$. Assume the claim holds for all trees with fewer than $n$ vertices and let $T$ be a tree with $n$ vertices. If $T$ is a path of length $n-1$, then by Example~\ref{ex:path},
$$F(T) = 2^{n-2} = \prod_{v\in T} d(v)!,$$
as desired. 

Suppose $T$ is not a path. Then there exists a vertex $v$ of degree $d\geq 3.$ Let $u_1,u_2,\ldots,u_d$ be the vertices adjacent to $v$ and write $|T_v(u_i)| = s_i$ for $i=1,2,\ldots,d$. Assume $s_1\leq s_2\leq \ldots \leq s_d.$ Let $T'$ be the subtree of $T$ obtained by removing all vertices in $T_v(u_1)$ and all edges incident to those vertices. Let $T''$ be the subtree of $T$ induced by edges in $E(T)\setminus E(T').$ See Figure~\ref{fig:theoremLowerBound} for illustration.

\begin{figure}[h]
\begin{tikzpicture}[scale=1]
\draw(0,0)--(1,0);
\draw(0,0)--(0,1);
\draw(0,0)--(0,-1);
\draw(0,0)--(-0.866,-0.5);
\draw(0,0)--(-0.866,0.5);

\node at (0,0){$\bullet$};
\node at (1,0){$\bullet$};
\node at (0,1){$\bullet$};
\node at (0,-1){$\bullet$};
\node at (-0.866,-0.5){$\bullet$};
\node at (-0.866,0.5){$\bullet$};

\draw (1.5,0) ellipse (2 and 0.5);
\draw (-0.7,0) ellipse (1.3 and 2);

\draw (0,0) node[below right]{$v$};
\draw (1,0) node[below]{$u_1$};
\draw (0,1) node[above]{$u_2$};
\draw (-0.866,0.5) node[above]{$u_3$};
\draw (0,-1) node[below]{$u_d$};
\draw (2,0) node[]{$T''$};
\draw (-1.2,0) node[]{$T'$};
\end{tikzpicture}
\caption{Merging a shelling of $T'$ and $T''_v$ to a shelling of $T$.}
\label{fig:theoremLowerBound}
\end{figure}
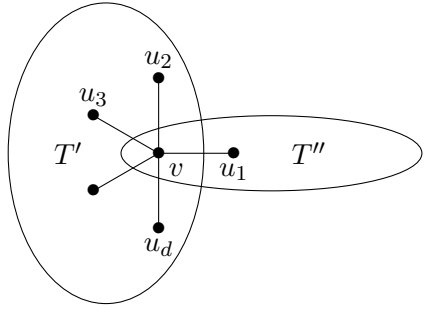

Suppose $\sigma'$ is a shelling of $T'$ and $\sigma''$ a shelling of $T''_v$. Merge $\sigma'$ and $\sigma''$ into a new permutation $\sigma$ of $E(T)$ such that (i) the order of edges in $\sigma'$ and in $\sigma''$ are preserved; (ii) $\sigma''(1) = (v,u_1)$ is not one of the first $s_d$ edges after merge. Note that $\sigma$ must be a shelling of $T$, since at least one of $\{\sigma'(k): 1\leq k\leq s_d\}$ is incident to $v$ and $(v,u_1)$ is adjacent to some previous edges in $\sigma$.

For each fixed $\sigma'$ and $\sigma''$, the number of $\sigma$ constructed by the above merging method is 
$$\binom{n-1-s_d}{|E(T'')|} = \binom{s_1+s_2+\cdots +s_{d-1}}{s_1}.$$
Therefore,
\begin{equation}\label{eq:shellingExtensionNumber}
F(T) \geq F(T')F(T''_v)\binom{s_1+s_2+\cdots+ s_{d-1}}{s_1}.
\end{equation}
Note that the shellings of $T$ constructed above do not include those whose first edge is $(v,u_1)$, so we can replace ``$\geq$" with ``$>$" in (\ref{eq:shellingExtensionNumber}).
Furthermore, by Lemma~\ref{lem:weightInequality},
$$F(T''_v) \geq \frac{F(T'')}{2^{s_1-1}}.$$
Note that $v$ has degree $d-1$ in $T'$ and degree 1 in $T''$, by the induction hypothesis,
$$F(T')F(T'') \geq \frac{1}{d}\prod_{u\in T} d(u)!.$$
Thus, (\ref{eq:shellingExtensionNumber}) implies
$$F(T) > \frac{1}{2^{s_1-1}d}\binom{s_1+s_2+\cdots+s_{d-1}}{s_1}\prod_{u\in T}d(u)!.$$

If for some choices of $v$ with degree $d\geq 3$, $\binom{s_1+s_2+\cdots+s_{d-1}}{s_1} \geq 2^{s_1-1}d$, then $F(T) > \prod_{u\in T} d(u)!$ and equality never holds.

If not, for all choices of $v$, $\binom{s_1+s_2+\cdots+s_{d-1}}{s_1} < 2^{s_1-1}d.$ 
By Lemma~\ref{app:ineq1} in appendix, $s_1 = s_2 =  \cdots = s_{d-1} = 1.$
Therefore, $T$ must be the following type of trees: for every vertex $v$ of degree $d(v) \geq 3$, it connects at least $d(v) - 1$ leaves. If $T$ is a star, then $F(T) = (n-1)!$ is an equality case. If not, $T$ has the form shown in Figure~\ref{fig:Stars}, where $v_0$ and $v_m$ are the only two possible vertices with degree at least 3. 

\begin{figure}[h]
\begin{tikzpicture}[scale=1]
\draw(0,0)--(2,0);
\draw(3,0)--(4,0);
\draw(0,0)--(0,1);
\draw(0,0)--(0,-1);
\draw(0,0)--(-0.707,0.707);
\draw(0,0)--(-1,0);
\draw(4,0)--(4,-1);
\draw(4,0)--(4,1);
\draw(4,0)--(5,0);
\draw(4,0)--(4.707,0.707);

\node[draw=none] (ellipsis2) at (2.5,0) {$\cdots$};
\node at (0,0){$\bullet$};
\node at (1,0){$\bullet$};
\node at (0,1){$\bullet$};
\node at (0,-1){$\bullet$};
\node at (-0.707,0.707){$\bullet$};
\node at (-1,0){$\bullet$};
\node at (4,0){$\bullet$};
\node at (5,0){$\bullet$};
\node at (4,1){$\bullet$};
\node at (4,-1){$\bullet$};
\node at (4.707,0.707){$\bullet$};
\node[rotate around={45:(0,0)}] (ellipsis1) at (-0.55,-0.45) {$\vdots$};
\node[rotate around={45:(0,0)}] (ellipsis1) at (4.5,-0.5) {$\cdots$};

\draw (0,0) node[below right]{$v_0$};
\draw (1,0) node[below]{$v_1$};
\draw (4,0) node[below left]{$v_m$};
\end{tikzpicture}
\caption{The only type of trees that satisfy case 2 condition.}
\label{fig:Stars}
\end{figure}
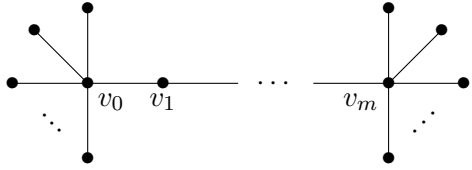 

Suppose $d(v_0) = d_1$, $d(v_m) = d_2$ with $2\leq d_1\leq d_2$. If $m=1$, by Proposition~\ref{prop:rootedTreeCounting},
\[
F(T_v) = \begin{cases}
\frac{(d_1+d_2-1)!}{d_2}, &\text{ if }v = v_0, \\
\frac{(d_1+d_2-1)!}{d_1}, &\text{ if }v = v_1, \\
\frac{(d_1+d_2-2)!}{d_2}, &\text{ if }v \neq v_1, \text{ and }(v,v_0)\in E(T), \\
\frac{(d_1+d_2-2)!}{d_1}, &\text{ if }v \neq v_0, \text{ and }(v,v_1)\in E(T).\\
\end{cases}
\]
Therefore, by Proposition~\ref{prop:rootedTreeCounting},
\begin{align*}
F(T) &= \frac{1}{2}\sum_{v\in T} F(T_v) \\
&= \frac{(d_1+d_2-2)!}{2}\left(\frac{d_1+d_2-1}{d_2}+\frac{d_1+d_2-1}{d_1} + \frac{d_1-1}{d_2}+\frac{d_2-1}{d_1}\right)\\
&=\frac{d_1^2+d_2^2+d_1d_2-d_1-d_2}{d_1d_2}(d_1+d_2-2)! \\
&\geq 2\cdot(d_1+d_2-2)!\\
&\geq d_1!d_2! = \prod_{u\in T} d(u)!,
\end{align*}
where the last two lines are due to Lemma~\ref{app:ineq2} in appendix. 
The equality holds only if $d_1 = d_2 = 2$ and $T$ is a single path.

Now suppose $m\geq 2.$ Consider the following type of shelling of $T$: The first $m-1$ edges of $\sigma$ consist of $\{(v_i,v_{i+1}): 0\leq i\leq m-2\}$. The number of shellings of such type is
$$2^{m-2}(d_1-1)!(d_2-1)!\binom{d_1+d_2-1}{d_2}.$$
Similarly, the number of shellings whose first $m-1$ edges consist of $\{(v_i,v_{i+1}): 1\leq i\leq m-1\}$ is
$$2^{m-2}(d_1-1)!(d_2-1)!\binom{d_1+d_2-1}{d_1}.$$
Thus, we have
\begin{align*}
F(T) &\geq 2^{m-2}(d_1-1)!(d_2-1)!\bigg[\binom{d_1+d_2-1}{d_2}+\binom{d_1+d_2-1}{d_1}\bigg] \\
&= 2^{m-2}(d_1-1)!(d_2-1)!\binom{d_1+d_2}{d_1}
\end{align*}
By Lemma~\ref{app:ineq3} in appendix,
\[
2^{m-2}(d_1-1)!(d_2-1)!\binom{d_1+d_2}{d_1}> 2^{m-1}d_1!d_2! = \prod_{u\in T}d(u)!
\]
unless $d_1=2, d_2\leq 4$, in which cases we have:
\begin{itemize}
\item $(d_1,d_2) = (2,2).$ $F(T) = 2^{n-2} = \prod_{u\in T}d(u)!.$ In this equality case, $T$ is a single path.
\item $(d_1,d_2) = (2,3).$ $F(T) = 2^{n-1}-2 > 3!\cdot 2^{n-4} = \prod_{u\in T}d(u)!.$
\item $(d_1,d_2) = (2,4).$ $F(T) = 6(2^{n-2}-n+1) > 4!\cdot 2^{n-5} = \prod_{u\in T}d(u)!.$
\end{itemize}

By induction, the proof of Theorem~\ref{thm:lowerBoundDegree} is complete.
\end{proof}

\begin{proof}[Proof of Theorem~\ref{thm:upperBoundDiameter}]
Let $v_0 - v_1 -\cdots - v_\ell$ be a longest path in $T$. Firstly, we reduce the problem to the case where all edges in $T$ are incident to $\{v_1,v_2,\ldots,v_{\ell-1}\}$. If not, construct a new tree $T'$ by removing every edge $e$ not incident to $\{v_i:1\leq i\leq \ell-1\}$ and adding a corresponding edge incident to $v_j$, where $v_j$ is the closest vertex from $e$ among $L$. Every shelling of $T$ is still a shelling of $T'$ by considering the corresponding edges. Thus, $F(T)\leq F(T')$ while the longest path remains the same.

Under this assumption, denote $V' = T\setminus \{v_0,v_1,\ldots,v_\ell\}.$ Consider the following operations:
\begin{enumerate}
\item[1.] Let $i$ be the smallest index such that $v_i$ has degree $\geq 3.$ If $i<\frac{\ell}{2}$, we remove all edges of the form $(v_i,u)$ for $u\in V'$ and add edges $(u,v_{i+1}).$
\item[2.] Repeat step 1 until no further operations can be performed.
\item[3.] Let $j$ be the largest index such that $v_j$ has degree $\geq 3.$ If $j>\frac{\ell}{2}$, we remove all edges of the form $(v_j,u)$ for $u\in V'$ and add edges $(u,v_{j-1}).$
\item[4.] Repeat step 3 until no further operations can be performed.
\end{enumerate}

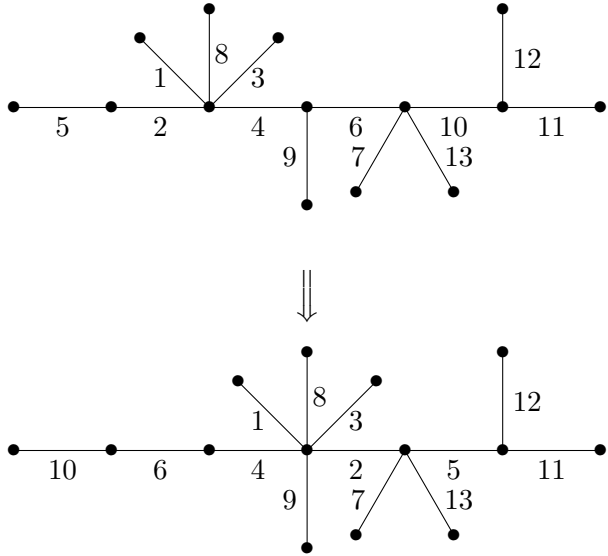
\begin{figure}[h]
\begin{tikzpicture}[scale=1.3]
\draw(1,0)--(7,0);
\draw(2.293,0.707)--(3,0)--(3.707,0.707);
\draw(3,0)--(3,1);
\draw(4,-1)--(4,0);
\draw(4.5,-0.866)--(5,0)--(5.5,-0.866);
\draw(6,0)--(6,1);
\node at (1,0){$\bullet$};
\node at (2,0){$\bullet$};
\node at (3,0){$\bullet$};
\node at (4,0){$\bullet$};
\node at (5,0){$\bullet$};
\node at (6,0){$\bullet$};
\node at (7,0){$\bullet$};
\node at (2.293,0.707){$\bullet$};
\node at (3,1){$\bullet$};
\node at (3.707,0.707){$\bullet$};
\node at (4,-1){$\bullet$};
\node at (4.5,-0.866){$\bullet$};
\node at (5.5,-0.866){$\bullet$};
\node at (6,1){$\bullet$};
\draw (1.5,0) node[below]{5};
\draw (2.5,0) node[below]{2};
\draw (3.5,0) node[below]{4};
\draw (4.5,0) node[below]{6};
\draw (5.5,0) node[below]{10};
\draw (6.5,0) node[below]{11};
\draw (2.5,0.5) node[below]{1};
\draw (2.95,0.55) node[right]{8};
\draw (3.5,0.5) node[below]{3};
\draw (4,-0.5) node[left]{9};
\draw (4.7,-0.5) node[left]{7};
\draw (5.3,-0.5) node[right]{13};
\draw (6,0.5) node[right]{12};
\end{tikzpicture}

$$\Big\Downarrow$$

\begin{tikzpicture}[scale=1.3]
\draw(1,0)--(7,0);
\draw(3.293,0.707)--(4,0)--(4.707,0.707);
\draw(4,0)--(4,1);
\draw(4,-1)--(4,0);
\draw(4.5,-0.866)--(5,0)--(5.5,-0.866);
\draw(6,0)--(6,1);
\node at (1,0){$\bullet$};
\node at (2,0){$\bullet$};
\node at (3,0){$\bullet$};
\node at (4,0){$\bullet$};
\node at (5,0){$\bullet$};
\node at (6,0){$\bullet$};
\node at (7,0){$\bullet$};
\node at (3.293,0.707){$\bullet$};
\node at (4,1){$\bullet$};
\node at (4.707,0.707){$\bullet$};
\node at (4,-1){$\bullet$};
\node at (4.5,-0.866){$\bullet$};
\node at (5.5,-0.866){$\bullet$};
\node at (6,1){$\bullet$};
\draw (1.5,0) node[below]{10};
\draw (2.5,0) node[below]{6};
\draw (3.5,0) node[below]{4};
\draw (4.5,0) node[below]{2};
\draw (5.5,0) node[below]{5};
\draw (6.5,0) node[below]{11};
\draw (3.5,0.5) node[below]{1};
\draw (3.95,0.55) node[right]{8};
\draw (4.5,0.5) node[below]{3};
\draw (4,-0.5) node[left]{9};
\draw (4.7,-0.5) node[left]{7};
\draw (5.3,-0.5) node[right]{13};
\draw (6,0.5) node[right]{12};

\end{tikzpicture}
\caption{An example of operation on $T$: moving edges towards middle. The shellings are indicated by the number on the edges. $g$ maps a shelling of the first tree to a shelling of the second tree.}
\label{fig:theoremUpperBound}
\end{figure}

Suppose the above operations end in step $M$. Let $T^{(t)}$ be the tree after $t^{\text{th}}$ operation. We claim that for all $t<M$,
\begin{equation*}
F(T^{(t+1)}) \geq F(T^{(t)}).
\end{equation*}
It suffices to prove the case when $t = 0.$ By symmetry, we can assume $i<\frac{\ell}{2}$. Let $V_i$ be the set of vertices adjacent to $v_i$ in $T$ except $v_{i-1}, v_{i+1}$. 
Define 
\begin{align*}
S_{T\cap T^{(1)}} &:= \{\sigma\text{ is a shelling of }T: \forall u\in V_i, (v_i,u) \text{ appears after }(v_i,v_{i+1}) \text{ in } \sigma\}, \\
S_{T\cap T^{(1)}}^{(1)} &:= \{\tau\text{ is a shelling of }T^{(1)}: \forall u\in V_i, (v_{i+1},u) \text{ appears after }(v_i,v_{i+1}) \text{ in } \tau\}, \\
S_{T\setminus T^{(1)}} &:= \{\sigma\text{ is a shelling of }T: \exists u\in V_i, (v_i,u) \text{ appears before }(v_i,v_{i+1}) \text{ in } \sigma\},\\
S_{T^{(1)}\setminus T}^{(1)} &:= \{\tau\text{ is a shelling of }T^{(1)}: \exists u\in V_i, (v_{i+1},u) \text{ appears before }(v_i,v_{i+1}) \text{ in } \tau\}.
\end{align*}

Note that there is a bijection between $S_{T\cap T^{(1)}}$ and $S_{T\cap T^{(1)}}^{(1)}$ by replacing edges of the form $(v_i,u)$ in every $\sigma\in S_{T\cap T^{(1)}}$ with $(v_{i+1},u)$, for all $u\in V_i$. Thus, $|S_{T\cap T^{(1)}}|=|S_{T\cap T^{(1)}}^{(1)}|$ and
$$F(T^{(1)}) - F(T) = |S_{T^{(1)}\setminus T}^{(1)}| - |S_{T\setminus T^{(1)}}|.$$
Define a function $g: S_{T\setminus T^{(1)}} \rightarrow S_{T^{(1)}\setminus T}^{(1)}$ as follows. Suppose $\sigma \in S_{T\setminus T^{(1)}}$ and denote $\tau = g(\sigma)$. If $\sigma(k) = (v_i, u)$ for some $u\in V_i$, $\tau(k) = (v_{i+1}, u)$; if $\sigma(k) = (v_j,u)$ for some $j\neq i$ and $u\in V'$, set $\tau(k) = (v_j,u)$. It remains to define $\tau(k)$'s where $\sigma(k)$ is an edge of $L$.

Write $e(j) = (v_{j},v_{j+1})$ for $j=0,1,\ldots,\ell-1.$ For each $r$ with $1\leq r\leq \ell$, suppose $\sigma(k_r) = e(j_r)$ where $k_1<k_2<\cdots < k_\ell.$ Define $\tau(k_r)$ inductively: when $r=1$, $\tau(k_1) = e(2i-j_1).$ When $r\geq 2$,
$$\tau(k_{r}) = \begin{cases}
e(j_{r}), & \text{ if } \{\tau(k_1), \tau(k_2),\ldots, \tau(k_{r-1})\} =\{e(j_1), e(j_2),\ldots, e(j_{r-1})\}  \\
e(2i-j_{r}), & \text{ otherwise.}
\end{cases}$$
The idea is that $g$ maps an edge not in $L$ to itself or to the corresponding edge in $T^{(1)}\setminus T$ (in the case that this edge is incident to $v_i$ in $T$). For edges in $L$, $g$
acts as a reflection with respect to $e(i)$, until the reflection image matches with the preimage. An example of $g$ is in Figure~\ref{fig:theoremUpperBound}.

We check the following properties of $g$:
\begin{itemize}
\item $g$ is well-defined. 

We first note that for any $r\leq \ell$, both $\{\sigma(k_1),\sigma(k_2),\ldots,\sigma(k_r)\}$ and $\{\tau(k_1),\tau(k_2),\ldots,\tau(k_r)\}$ form a path $P_r$ and $P^{(1)}_r$ in $L$, respectively. Since $j_1 \leq i$, the right endpoint of $P^{(1)}_r$ is never on the left side of the right endpoint of $P_r$ (assuming that $L$ is a horizontal path with left endpoint $v_0$ and right endpoint $v_\ell$, as illustrated in Figure~\ref{fig:theoremUpperBound}). Furthermore, since the ``branching edges" of $T$ (edges in $E(T)\setminus E(L)$) are not on the left side of $v_i$, every branching edge adjacent to $P_r$ must be adjacent to $P^{(1)}_r$. Thus, $\tau$ is a shelling of $T^{(1)}$. Moreover, $\tau\in S_{T^{(1)}\setminus T}^{(1)}$ by the correspondence between $(v_i,u)\in \sigma$ and $(v_{i+1},u)\in \tau$ for all $u\in V_i$. Therefore, $g$ is well-defined.

\item $g$ is injective.

Suppose $g(\sigma) = \tau.$ By the definition of $g$, $\sigma(k)$ is uniquely determined whenever $\tau(k) \not\in L.$ Suppose $\tau(k_r) = e(i_r)$ for $1\leq r\leq \ell$. we can recover $\sigma(k_r)$ from $\tau$: $\sigma(k_1) = e(2i-i_1)$. When $r\geq 2$,
$$\sigma(k_{r}) = \begin{cases}
e(i_{r}), & \text{ if } \{\sigma(k_1), \sigma(k_2),\ldots, \sigma(k_{r-1})\} =\{e(i_1), e(i_2),\ldots, e(i_{r-1})\}  \\
e(2i-i_{r}), & \text{ otherwise.}
\end{cases}$$
Therefore, $\sigma$ is uniquely determined by $\tau$ and $g$ is injective.
\end{itemize}

Since $g$ is injective, $|S_{T^{(1)}\setminus T}^{(1)}| \geq |S_{T\setminus T^{(1)}}|$ and thus $F(T^{(1)}) \geq F(T).$

Finally, note that $T^{(M)}$ is the tree where all edges not in $L$ are incident to $v_{\lfloor\frac{\ell}{2}\rfloor}$. By Proposition~\ref{prop:rootedTreeCounting} and ~\ref{prop:treeCounting},
$$F(T^{(M)})= 
\begin{cases}
\frac{2(n-1-\frac{\ell}{2})!}{(\frac{\ell}{2})!}\bigg[\binom{n-2}{\frac{\ell}{2}}+\sum_{i=0}^{\frac{\ell}{2}-1}\binom{n-1}{i}\bigg], &\text{ if }\ell \text{ is even,} \\
\frac{(n-\frac{\ell+3}{2})!}{(\frac{\ell+1}{2})!}\bigg[(n-1-\ell)\binom{n-2}{\frac{\ell-1}{2}}+n\sum_{i=0}^{\frac{\ell-1}{2}}\binom{n-1}{i}\bigg], &\text{ if } \ell \text{ is odd.}
\end{cases}
$$
Thus, the proof of inequality is complete.

Futhermore, we shall prove that $g$ is surjective only if $T$ is isomorphic to $T^{(M)}.$ If not, then there are two cases:

\noindent \textbf{Case 1.} $i< \frac{\ell-1}{2}$.

In this case, $2i<\ell-1$. Thus, for every $\sigma \in S_{T\setminus T^{(1)}}$, $g(\sigma)(1)\neq e(\ell-1) = (v_{\ell-1},v_\ell)$. However, there exists $\tau\in S_{T^{(1)}\setminus T}^{(1)}$ whose first edge is $(v_{\ell-1},v_{\ell})$, contradiction!

\noindent \textbf{Case 2.} $i= \frac{\ell-1}{2}$ and there exists another vertex $v_j$ of degree at least 3.

Suppose $(v_j, u)$ is an edge not in $L$, then for every $\sigma \in S_{T\setminus T^{(1)}}$, $g(\sigma)(1)$ cannot be this edge. However, there exists $\tau\in S_{T^{(1)}\setminus T}^{(1)}$ whose first edge is $(v_j,u)$, contradiction!

Therefore, $g$ is surjective only if $T$ is isomorphic to $T^{(M)}$, so $$F(T) = F(T^{(M)})$$ if and only if $T$ is isormorphic to $T^{(M)}$. This completes the proof of Theorem~\ref{thm:upperBoundDiameter}.
\end{proof}

\section*{acknowledgements}
This research was carried out as part of the 2018 Summer Program in Undergraduate Research (SPUR) of the MIT Mathematics Department. The authors would like to thank Prof.\,Richard Stanley for suggesting the project and Prof.\,Alex Postnikov, Prof.\,Ankur Moitra and Prof.\,Davesh Maulik for helpful conversations. The authors also thank Christian Gaetz and the anonymous referee for pointing out the connection between the shelling number of trees and the number of linear extensions of tree posets.

\appendix
\section{Some Combinatorial Inequalities}
\label{app:inequalities}
\begin{lemma}\label{app:ineq1}
Let $s_1\leq s_2 \leq \cdots \leq s_{d-1}$ and $d\geq 3$ be some positive integers. Then
$$\binom{s_1+s_2+\cdots+s_{d-1}}{s_1} < 2^{s_1-1}d$$
if and only if $s_1=s_2=\cdots = s_{d-1}= 1.$
\end{lemma}
\begin{proof}
Note that $s_1+s_2+\cdots + s_{d-1} \geq (d-1)s_1$, so $$\binom{s_1+s_2+\cdots+s_{d-1}}{s_1}\geq \binom{(d-1)s_1}{s_1}.$$
We claim that when $s_1\geq 2,$
$$\binom{(d-1)s_1}{s_1} \geq 2^{s_1-1}d.$$
Induct on $d$. When $d = 3$,
$$\binom{2s_1}{s_1} = \prod_{k=1}^{s_1}\frac{s_1+k}{k}\geq (s_1+1)\prod_{k=2}^{s_1}\frac{s_1+k}{k} \geq 3\cdot 2^{s_1-1}.$$
Suppose that the claim holds for $d-1$, then
\begin{align*}
\binom{ds_1}{s_1} = \binom{(d-1)s_1}{s_1}\prod_{k=1}^{s_1}\frac{(d-1)s_1+k}{(d-2)s_1+k} \geq (2^{s_1-1}d)\cdot \frac{d}{d-1} \geq 2^{s_1-1}(d+1),
\end{align*}
so the claim is proved by induction.

According to this claim, if $s_1\geq 2$,
$$\binom{s_1+s_2+\cdots+s_{d-1}}{s_1}\geq 2^{s_1-1}d,$$
contradiction! So $s_1 = 1$ and $s_1+s_2+\cdots + s_{d-1} < d$. This gives $s_1=s_2=\cdots = s_{d-1} = 1.$
\end{proof}

\begin{lemma}\label{app:ineq2}
Suppose $2\leq d_1\leq d_2$ are positive integers, then
$$2\cdot (d_1+d_2-2)! \geq d_1!d_2!.$$
\end{lemma}
\begin{proof}
Note that 
$$\frac{(d_1+d_2-2)!}{d_2!} = \prod_{k=1}^{d_1-2}(d_2+k) \geq \prod_{k=1}^{d_1-2}(2+k) = \frac{d_1!}{2},$$
so the lemma follows immediately.
\end{proof}

\begin{lemma}\label{app:ineq3}
Suppose $2\leq d_1\leq d_2$ are positive integers, then
$$\binom{d_1+d_2}{d_1}\leq 2d_1d_2$$
if and only if $d_1=2$ and $d_2\leq 4$.
\end{lemma}
\begin{proof}
We claim that when $d_1 \geq 3,$
$$\binom{d_1+d_2}{d_1}> 2d_1d_2.$$

Induct on $d_1$. When $d_1 = 3$, 
$$\binom{d_1+d_2}{d_1}-2d_1d_2 = \frac{(d_2+3)(d_2+2)(d_2+1)}{6}-6d_2 = f(d_2).$$
If $d_2 = 3$, $f(d_2) = 2 > 0$.
If $d_2 \geq 4,$
$$f(d_2) \geq \frac{(4+3)(4+2)(d_2+1)}{6}-6d_2 = d_2+7>0.$$
Suppose that the claim holds for $d_1-1$, then
$$\binom{d_1+d_2}{d_1} = \frac{d_1+d_2}{d_1} \binom{d_1+d_2-1}{d_1-1} >  \frac{d_1+d_2}{d_1} 2(d_1-1)d_2 > 2d_1d_2,$$
so the induction is complete.

According to this claim, $d_1 = 2$ and
$$\binom{2+d_2}{2}>4d_2.$$ This implies $$d_2^2-5d_2 + 2 \leq 0$$
and thus $d_2 \leq 4$.
\end{proof}

%
%


\end{document}